\documentclass{amsart}
\usepackage{graphicx} 
\usepackage{amsmath}
\usepackage{amssymb}

\usepackage{amsthm}
\usepackage{graphicx}
\usepackage{amscd}
\usepackage{epic, eepic}
\usepackage{url}
\usepackage{xcolor}
\usepackage{todonotes}
\usepackage{enumerate}
\usepackage{enumitem}
\usepackage{comment}
\usepackage{hyperref}
\usepackage[utf8]{inputenc}
\usepackage{tabularray}

\newtheorem{theorem}{\bf Theorem}[section]
\newtheorem{lemma}[theorem]{\bf Lemma}

\newtheorem{obs}[theorem]{\bf Observation}

\theoremstyle{definition}

\newtheorem{definition}[theorem]{\bf Definition}
\def\red{\textcolor{red}}

\newcommand{\Z}{\mathbb{Z}}
\newcommand{\Q}{\mathbb{Q}}

\newcommand{\C}{\mathbb{C}}
\newcommand{\R}{\mathbb{R}}

          \usepackage{array}

\usepackage{todonotes}

\title{A lonely weak tile}
\author{Gergely Kiss, Itay Londner, M\'at\'e Matolcsi, G\'abor Somlai}

\begin{document}

\begin{abstract}
The notion of weak tiling was a key ingredient in the proof of Fuglede's spectral set conjecture for convex bodies, due to the fact that every spectral set tiles its complement weakly with a suitable Borel measure. In this paper we review the concept of weak tiling, and answer a question raised by Kolountzakis, Lev and Matolcsi, by giving an example of a set $T$ which tiles its complement weakly, but $T$ is neither spectral, nor a proper tile. 
    
\end{abstract}
\maketitle

\begin{center}
{\it In memory of Bent Fuglede} 
\end{center}

\section{Introduction}

In 1974, while studying commutation properties of partial differential operators $-i \partial/\partial x_j$, Bent Fuglede introduced the notion of spectral sets, and formulated his "spectral set conjecture" \cite{fug}. In the decades to follow, extensive research on this conjecture has unveiled a wealth of connections between harmonic analysis, geometry, combinatorics and number theory. 

\medskip

In this paper we will review some of the recent developments in connection with Fuglede's conjecture, and answer a question raised recently in \cite{weak}. In particular, we will focus on the concept of {\it weak tiling}, and present a set $T$ which tiles its complement weakly, but $T$ is neither spectral, nor a proper tile. In this sense we can call $T$ a "lonely weak tile", because neither does it have a spectrum, nor does it have a tiling complement. 

\medskip

The paper is organized as follows. In Section \ref{sec2} we recall the necessary notions, introduce notations and formulate the problem. In Section \ref{sec3} we present our construction of a lonely weak tile.

\section{Spectral sets, tiling and weak tiling}\label{sec2}

Let $E$ denote a locally compact Abelian (LCA) group, with Haar measure $\nu$. The character group of $E$ (the Pontryagin dual) will be denoted by $\hat{E}$. Recall that a measurable  set $X\subset E$ of finite positive measure is called a {\it tile} if some essentially disjoint translated copies of $X$ cover $E$ unilaterally, i.e. there exists a translation set $\Lambda\subset E$ such that $\cup_{\lambda\in\Lambda} X+\lambda=E$, up to measure 0, and the union is disjoint up to measure zero. In notation we write $X+\Lambda =E$, or in an analytic manner, $1_X\ast 1_\Lambda=1_E$ (almost everywhere), where $1_X$ is the indicator function of $X$, and $1_\Lambda$ is the sum of Dirac masses put on each $\lambda\in \Lambda$.  

\medskip

A measurable set $X\subset E$ of finite positive measure is called {\it spectral} if there exists a set of characters $S\subset \hat{E}$ such that the restrictions of the elements of $S$ to $X$ form an orthogonal basis in $L^2(X)$. Such a set $S$ is called a {\it spectrum} of $X$. 

\medskip

Fuglede's spectral set conjecture stated that a set $X\subset E$ is spectral if and only if it is a tile. In his original paper \cite{fug} Fuglede stated the conjecture in the case $E=\R^d$, but he remarked that the notions make sense in any LCA group, and experimented with some particular cases in finite cyclic groups.  

\medskip

The conjecture was initially proved in several particular cases \cite{laba1, ikp, ikt1, ikt2}, but was eventually disproved in $\R^d$ for any $d\ge 3$ by a series of counterexamples  \cite{tao, fug4, had, tnos, fr, fmm}.  However, in a recent breakthrough \cite{conv}, the conjecture was proved for convex bodies in all dimensions $d$. 

\medskip

One of the key concepts in the proof of \cite{conv} was {\it weak tiling}. Recall, that a bounded measurable set $X\subset E$ weakly tiles another (possibly unbounded)  set $D\subset E$ if there exists a locally finite Borel measure $\mu$ such that $1_X \ast \mu =1_D$. The main point in the argument of \cite{conv} is that if $X$ is spectral then $X$ tiles its own complement $X^c$ weakly. 

\medskip

It is also clear that if $X$ is a tile, then $X$ also tiles $X^c$ weakly. Therefore, the property of a set tiling its complement weakly is shared by both spectral sets and tiles. 

\medskip

Interestingly, it is not at all  obvious to come up with an example of a set $X$ which tiles its complement weakly, but does not tile the group $E$ properly. In fact, the proof of Fuglede's conjecture for convex bodies in \cite{conv} relies on a geometric argument that shows that such a situation cannot happen if $X$ is a convex body. 

\medskip

The natural question arises: are there any such examples? The  answer is affirmative, because one can consider any set $X$ which is spectral, but not a tile (such sets exist, and are constructed explicitly \cite{tao, had}, in certain finite groups, and in $\R^d$ in any dimension $d\ge 3$).

\medskip

The following natural question was raised in \cite{weak}: is there an example of a set $X\subset \R^d$ such that $X$ tiles $X^c$ weakly but $X$ is neither spectral nor a tile?  In Section \ref{sec3} we will answer this question in the affirmative by constructing such a set. The construction will be carried out in a finite group $E$, but we will then describe how the example can be lifted to $\R^5$ by the standard techniques described in \cite{tao, tnos}. The arising set will be a finite union of unit cubes, and hence can be considered as open or closed, as desired. Also, the techniques of \cite{gk} can be applied to make the example connected (in a higher dimensional space $\R^d$).

\section{A lonely weak tile}\label{sec3}

As usual, for any $n$, the finite cyclic group of order $n$ will be denoted by $\Z_n$. Also, for a finite Abelian group $E$, and any function $f: E\to \C$ we will use the Fourier transform $\hat{f}$ of $f$ defined by $\hat{f}(\gamma)=\sum_{x\in E} f(x)\gamma(x)$, for any $\gamma\in \hat{E}$. The Dirac mass at a point $x\in E$ will be denoted by $\delta_x$.

\medskip

As we shall see, the set we construct below will exhibit a slightly stronger property than weakly tiling its own complement. We recall the definition of weak pd-tiling from \cite{kms}.

\begin{definition}
For a finite Abelian group $E$, we say that a set $X\subset E$ pd-tiles $E$ weakly if there exists a nonnegative function $h: E\to \R$ such that $h(0)=1$, $1_X\ast h=1_E$, and its Fourier transform $\hat{h}\ge 0$. To abbreviate the terminology, we will sometimes just say that $X$ pd-tiles $E$, dropping the term "weakly".    
\end{definition}

The letters "pd" stand for "positive definite". Clearly, if $X$ pd-tiles $E$ then $X$ tiles $X^c$ weakly (this is the significance of $h(0)=1$). 

\medskip

Recall from \cite{kms} that if $X$ is a tile, or is a spectral set, then $X$ pd-tiles $E$. In what follows, we will construct a set $X$ in an appropriate finite group $E$ such that $X$ pd-tiles $E$ but $X$ is neither spectral, nor a tile. 

\medskip

We begin our construction by recalling some results from the literature.

\medskip

For every odd prime $p$, there exists a spectral set $B$ in the finite group $H=\Z_p^4$, which is of cardinality $2p$, and therefore does not tile the group \cite{nontilep}. We can fix any $p>3$. 

\medskip

Next, let $G_1=\Z_6^5$, $G_2=\Z_q$, and consider the group $G=G_1\times G_2$. As explained in \cite{tnos}, there exists a non-spectral tile $A\subset G$ for any $q\ge 15$. We will choose $q$ to be a large prime, $q>6 p^4$. We recall the construction of $A$ because we will need its properties. We will identify elements of $G$ with vectors of length 5+1, where the first 5 coordinates range from 0 to 5, and the last coordinate ranges from 0 to $q-1$. 

Let $\pi\in S_5$ be any permutation, and  $v=(1,2,3,4,5)\in \Z_6^5$. Let $\phi_\pi : \Z_6^5\to \Z_6$ be the homomorphism defined by $\phi_\pi (x)=\langle \pi(v), x \rangle$ mod 6 (where $\langle \cdot, \cdot \rangle$ is simply the dot product of the vectors). Let $A_\pi =Ker \ \phi_\pi$. It is important to note for later purposes that all $A_{\pi}$ are subgroups. We enumerate the 120 elements of $S_5$ as $ \pi_k, 0\leq k\leq 119$. The set $A$ is constructed layer by layer: for $0\le k\le 119$ consider the set $A_k=A_{\pi_k}\times\{ k\}\subset \Z_6^5 \times \Z_q$, where $k$ indicates the $\Z_q$ coordinate. For $120\le k$ we can take $A_k=A_{\pi_0}\times \{k\}$ (on the top layers we just keep repeating $A_{\pi_0}$). Finally, define $A=\cup_{k=0}^{q-1}A_k$. As explained in \cite{tnos} the set $A$ tiles $G$, but it is not spectral. 

\medskip

It is now natural to consider the direct product group $E=G\times H$, and  $A\times B\subset E$, in the hope that it pd-tiles $E$ but it is neither spectral, nor a tile. We do not know whether this is indeed the case, because we are not able to decide whether $A\times B$ is spectral or not. For this reason, we need to modify the set $A\times B$ slightly. 

\medskip

We need to introduce some terminology concerning the product group $G\times H$. If $u:G\to \C$ and $v: H\to \C$ are arbitrary functions, we define $u\otimes v: G\times H\to \C$ by the formula $u
\otimes v \ (g,h)=u(g)v(h)$. Notice that $u\otimes v$ can be written as a convolution, $u\otimes v= (u\otimes \delta_{0_H}) \ast (\delta_{0_G}\otimes v$). Note also, that $\widehat{u\otimes v}=\hat{u}\otimes  \hat{v}$. In particular, if both $u$ and $v$ are positive definite, then so is $u\otimes v$. 

\medskip

We now introduce some flexibility to the set $A\times B$ by shifting the copies of $B$ in the $H$-component. Formally, for every map $t:A\rightarrow H$ let 
\begin{equation}\label{pt}
P_t= \bigcup_{a\in A} \{a\}\times\{t(a)+B\}    
\end{equation}

When $t$ is the constant 0 mapping, we get back $P_0=A\times B$. 

\begin{lemma}
For every map $t: A \to H$, the set $P_t$ defined in equation \eqref{pt} is not a tile in $G\times H$. 
\end{lemma}
\begin{proof}
Assume $P_t$ is a tile, by contradiction. Then $\cup_{m\in M} (P_t+m)=G\times H$ for some translation set $M$. Consider those elements $m\in M$ for which $K_m:=(P_t+m)\cap (\{0\}\times H)\ne \emptyset$. Each such $K_m$ is a translated copy of $B$, i.e. $K_m=\{0\}\times (x_m+B)$. This means that 
$$\bigcup_{m: K_m\ne \emptyset} (x_m+B)=H$$ is a tiling. However, we know that $B$ cannot tile $H$, a contradiction. 
\end{proof}

We remark that the construction of $P_t$ and the statement of  Lemma 3.2 above are similar to those of  \cite{gk}[Lemma 3.1].

\begin{lemma}
For every map $t: A \to H$, the set $P_t$ defined in equation \eqref{pt} pd-tiles $G\times H$. 
\end{lemma}
\begin{proof}
Let $1_A\ast w_A=1_G$ and $1_B\ast w_b=1_H$ be pd-tilings of $G$ and $H$ by $A$ and $B$, respectively. These exist because $A$ is a tile and $B$ is spectral (see \cite{kms}). Then, for every $h \in H$, we also have $1_{h+B}\ast w_B=1_H$, because $1_{h+B}=\delta_h \ast 1_B$. Therefore, $P_t\ast  (\delta_{0_G}\otimes w_B)$ is constant 1 on the cosets $a+H$ for all $a\in A$, and constant zero on $g+H$ for all $g \in G\setminus A$. Hence,  
\begin{equation}
P_t\ast \left( \delta_{0_G}\otimes w_B \right) \ast (w_A\otimes \delta_{0_H})=P_t \ast (w_A\otimes w_B)=1_{G\times H}
\end{equation}
 is a pd-tiling of $G\times H$ by $P_t$. 
\end{proof}

We now specify the mapping $t: A\to H$. Fix a basis $\{v_1, v_2, v_3, v_4\}$ of the vector space $\Z_p^4$. Let $t$ be defined as 

\begin{equation}\label{t1}
t(0_{G_1},k)=v_k \ \ \text{for} \ \ 1\le k \le 4, \ \text{and} \ \ t(a, j)=0_H \ \ \text{otherwise.}
\end{equation}

(Note here that $a_k:=(0_{G_1},k)\in A_k$ for every $0\le k\le q-1$ because $A_k=A_{\pi_k}\times \{k\}$ and $A_{\pi_k}$ is a subgroup.) In plain terms, we just shift four copies of $B$ in comparison to the direct product $A\times B$, and the shifting vectors form a basis of $\Z_p^4$. 

\medskip

Consider the set $P_t$ corresponding to the mapping above. We will need to analyze the zero-set of the Fourier transform $\hat{1}_{P_t}$. Note that, by definition, 
\begin{equation}\label{ft}
\hat{1}_{P_t} (\gamma, \rho)=
\hat{1}_B(\rho)\sum_{a\in A} \gamma(a)\rho(t_a).
\end{equation}

For the sake of concrete calculations, it will be convenient to identify any element $z$ of $E=G\times H=G_1\times G_2 \times H$ with a vector $z=(u_1, u_2 , v)$ where $u_1, u_2, v$ are vectors of length 5, 1 and 4, respectively, and the coordinates of $u_1$ range from 0 to 5, the coordinate of $u_2$ ranges from 0 to $q-1$, and the coordinates of $v$ range from 0 to $p-1$. The components of $u_1, u_2, v$ will be referred to as the $G_1$-component, $G_2$-component and $H$-component of $z$. Sometimes we group $u_1$ and $u_2$ as $(u_1, u_2)$, and refer to it as the $G$-component of $z$. 
An element $\alpha$ of $\hat{E}$ will also be identified with such a vector $\alpha=(\gamma_1, \gamma_2, \rho)$, and the action of $\alpha$ on $z$ will be given by $\alpha(z)=e^{2i\pi (\langle \gamma_1, u_1 \rangle/6 +  \gamma_2 u_2 /q + \langle \rho, v \rangle /p)}=\omega_6^{\langle \gamma_1, u_1 \rangle} \omega_q^{\gamma_2 u_2}\omega_p^{\langle \rho, v \rangle}$, where $\omega_n=e^{2i\pi/n}$ is an $n$th root of unity. 

\medskip

Let us now investigate the expression \eqref{ft}. Note first, that the sum can be written as $\sum_{a\in A} \gamma(a)\rho(t_a)=\hat{1}_A(\gamma)-\sum_{i=1}^4 \gamma(a_i)(1-\rho(v_i))$, and therefore 

\begin{equation}
\hat{1}_{P_t} (\gamma, \rho)=
\hat{1}_B(\rho) \left (\hat{1}_A(\gamma)-\sum_{i=1}^4 \gamma(a_i)(1-\rho(v_i))\right ).   
\end{equation}

We can identify some trivial cases where $\hat{1}_{P_t} (\gamma, \rho)=0$. Namely, this is the case when $\hat{1}_B(\rho)=0$. When $\hat{1}_B(\rho)\ne 0$ and $\rho=0$, the expression for $\hat{1}_{P_t} (\gamma, \rho)$ simplifies to $\hat{1}_B(\rho) \hat{1}_A(\gamma)$, and it is zero if and only if $\hat{1}_A(\gamma)=0$. The remaining case, when $\hat{1}_B(\rho)\ne 0$ and $\rho\ne 0$, is non-trivial, but we have the following important partial information on it. 

\begin{lemma}\label{qn0}
Let $P_t$ be defined by equations  \eqref{pt} and \eqref{t1}. Assume that $\hat{1}_B(\rho)\ne 0$, and $\rho\ne 0$, and the $\Z_q$-component of $\gamma=(\gamma_1, \gamma_2)$ satisfies $\gamma_2\ne 0$. Then $\hat{1}_{P_t} (\gamma, \rho)\ne 0$
\end{lemma}
\begin{proof}
	We will need the following fact from elementary Galois theory: if $\sum_{j=0}^{q-1} \beta_j\omega_q^j=0$ where $\beta_j\in \Q(\omega_{6p})$, then all $\beta_j$ are necessarily equal. (Here we use the important property of our construction that $6p$ and $q$ are relatively prime to each other, and hence the minimal polynomial of $\omega_q$  over $\Q(\omega_{6p})$ is $\Phi_q(x)=\sum_{j=0}^{q-1}x^j$; see \cite{kang})  
		
		\medskip
		
		We need to prove that 
		$\hat{1}_A(\gamma)-\sum_{k=1}^4 \gamma(a_k)(1-\rho(v_k))\ne 0$. Since $a_k=(0_{G_1}, k)$, we have $\gamma(a_k)=\omega_q^{k \gamma_2}$.
		
		\medskip

		Note first, that 
		\begin{align}
			\hat{1}_A(\gamma) & =\hat{1}_A(\gamma_1, \gamma_2)=\\
			&\sum_{(u_1,u_2)\in A} \omega_6^{\langle \gamma_1, u_1 \rangle}\omega_q^{\gamma_2 u_2}=\\
			&	\label{eq7}	\sum_{u_2=0}^{q-1}  \omega_q^{\gamma_2 u_2} \left (\sum_{u_1\in A_{u_2}} \omega_6^{\langle \gamma_1, u_1\rangle} \right).
		\end{align}
		
		Due to the fact that $\gamma_2\ne 0$ the exponents $\gamma_2 u_2$ range through 0 to $q-1$ (mod $q$). Recall that each slice  $A_i, 0\le i\le q-1$ is subgroup of size $6^4$. Since the inner sum in \eqref{eq7} is the evaluation of the Fourier transform of its indicator function, which is an indicator function of a subgroup, it may only attain two values $0$ and $6^4$.
		
		
		\medskip
		
		There are three essentially different cases to consider. 
		
		\medskip
		
		The first is when all inner sums are zero. This happens if and only if $\gamma_1\ne (0,0,0,0,0)$ and $\gamma_1\ne c\cdot \pi(1,2,3,4,5)$ for any $c\ne 0$ and $\pi\in S_5$. 
		
		\medskip
		
		The second is when all inner sums are $6^4$. This happens if and only if $\gamma_1=(0,0,0,0,0)$. 
		
		\medskip
		
		The third is when some inner sums are 0, while others are $6^4$. This happens for all the remaining vectors $\gamma_1$ not covered by the first two cases, i.e when $\gamma_1= c\cdot \pi(1,2,3,4,5)$ for some $c\ne 0$ and $\pi\in S_5$ (see \cite{tnos}).
		
		\medskip
		
		In the first two cases, $\hat{1}_A(\gamma)=0$, and 
		$\hat{1}_A(\gamma)-\sum_{k=1}^4 \gamma(a_k)(1-\rho(v_k))= -\sum_{k=1}^4 \omega_q^{k \gamma_2}(1-\rho(v_k))$. We can view it as a sum $\sum_{k=0}^{q-1} \beta_k w_q^{k\gamma_2}$ where $\beta_k=0$ for all $k\ne 1, 2, 3, 4$. Such an expression could only be zero if all $\beta_k=0$. However, this is not the  case because $v_1, v_2, v_3, v_4$ form a basis of $\Z_p^4$ and  $\rho\ne (0,0,0,0)$,  hence there exists an index $1\le k\le 4$ such that $\langle v_k , \rho \rangle \ne 0$, and $\rho(v_k)=\omega_p^{\langle \rho , v_k \rangle}\ne 1$.

		
		\medskip
		
		In the third case, we can pick an exponent $j_1=\gamma_2 u_2^{(1)}$ such that the coefficient of $\omega_q^{j_1}$ in $\hat{1}_A$ is 0, and another exponent $j_2=\gamma_2 u_2^{(2)}$ where the coefficient of $\omega_q^{j_2}$ is $6^4$. It is then clear that the coefficients of $\omega_q^{j_1}$ and $\omega_q^{j_2}$ in the expression $\hat{1}_A(\gamma)-\sum_{k=1}^4 \gamma(a_k)(1-\rho(v_k))$ cannot be equal because $|1-\rho(v_k)|\le 2$. Therefore, $\hat{1}_A(\gamma)-\sum_{k=1}^4 \gamma(a_k)(1-\rho(v_k))\ne 0$ once again. 
	\end{proof}
	
We are now in a position to prove that $P_t$ is not spectral. 
	
	\begin{lemma}
		The set $P_t$ defined by equations \eqref{pt} and \eqref{t1} is not spectral in $G\times H$. 
	\end{lemma}
	
	\begin{proof}
		
		Let us assume, by contradiction,  that $P_t$ is spectral, and let $S\subset \hat{G}\times \hat{H}$ be a spectrum of it. Then 
  \begin{equation}\label{spsize}
  |S|=(6^4q)(2p).
  \end{equation}

		\medskip
		
		Let $V$ denote the set of elements $\rho\in \hat{H}$  such that $(\hat{G}\times \{\rho\}) \cap S \ne \emptyset$.
		We define a graph $\Gamma$ whose vertex set is $V$. Two vertices $\rho_1 \ne \rho_2$ are connected by an edge if and only if $\hat{1}_B (\rho_1-\rho_2) \ne 0$. We denote the set of isolated points of $\Gamma$ by $Y$.
		
		\medskip
		
		Notice that the cardinality of $Y$ is at most $2p$. This is because the isolated characters are mutually orthogonal over $B$, and $|B|=2p$. 
		
		\medskip
		
		Notice also that if a vertex $\rho  \in V$ is not isolated, then $|(\hat{G}\times \{\rho\}) \cap S| \le 6^5$. This follows from Lemma \ref{qn0}. Indeed, assume $\rho_1$ is connected to $\rho$, and pick an element $s_1=((\gamma_1, \gamma_2), \rho_1)\in S$. Then for each  $s'=((\gamma_1', \gamma_2'), \rho)\in S$ on the layer $\rho$, we have $\hat{1}_B(\rho_1-\rho)\ne 0$, $\rho_1-\rho\ne 0$ and $\hat{1}_{P_t}(s-s')=0$, so we must have $\gamma_2=\gamma_2'$. There are at most $6^5$ elements with a fixed $\Z_q$-component. 
		
		\medskip
		
		Finally we perform a case by case analysis of the structure of the set $V$. 
		
		\medskip
		
		Let us assume first that $|V| \le 2p$. Then there exists $\rho \in V$ such that $S_\rho :=(\hat{G}\times \{\rho\})  \cap S$ is of cardinality at least $6^4 q$. However, for any $s=(\gamma, \rho)$, $s'=(\gamma', \rho)\in S_\rho$ the $\hat{H}$-component of $s-s'$ is 0, so we have $\hat{1}_{P_t}(s-s')=|B|\hat{1}_A(\gamma -\gamma')$ by \eqref{ft}, and it must be 0 due to $S$ being a spectrum of $P_t$. This contradicts the fact that $A$ is not spectral in $G$. 
		
		\medskip
		
		Let us now assume that $|Y|<2p$. Layers of isolated points contain at most $6^4q-1$  elements of $S$ (due to the fact that $A$ is not spectral in $G$, as explained above), while layers of non-isolated points contain at most $6^5$ elements of $S$. Therefore, 
		\begin{align*}
			|S| & \leq |Y|(6^4 q -1)+(p^4-|Y|) 6^5\\
			&\leq (2p-1)(6^4 q -1)+p^4 \cdot 6^5\\
			&<2p \cdot 6^4 \cdot q,
		\end{align*}
		since $q> 6p^4$, a contradiction to \eqref{spsize}.
	
		\medskip

		Let us finally assume that $|Y|=2p$, and $|V|>2p$. Then, for a vertex $v \in  V \setminus Y$, we have that $Y\cup \{v\}$ is an independent set (i.e. a set in which no two elements are connected by an edge) of size $2p+1$ in $V$. This again means that the corresponding characters are pairwise orthogonal on $B$, a contradiction. 
		
	\end{proof}

Finally, let us describe how a lonely weak tile can be constructed in $\R^5$ based on the example above. 

\medskip

Recall that $p>3$ and $q>6p^4$ are primes. The group $G\times H$ is isomorphic to $\Z_{6p}^4\times \Z_{6q}$, and can be identified with the set 
$\{0, 1, \dots,6p-1\}^4\times \{0, 1, \dots, 6q-1\}$ which, in turn, is a subset of $\Z^5$  (a discrete box). The set $P_t$ constructed by equations \eqref{pt} and \eqref{t1} can be identified with a subset $P$ of this box. For every $k\ge 1$ we can consider in $\Z^5$ the set $T(k)=\{0, 6p, \dots (k-1)6p\}^4\times \{0, 6q, \dots, (k-1)6q\}$, 
and the set $P(k)$ defined as $P(k)=P+T(k)$. By the results of \cite{tao} and \cite{tnos}, the set $P(k)$ is neither a tile, nor spectral in $\Z^5$, if $k$ is large enough. On the other hand, $P(k)$ pd-tiles $\Z^5$ in an obvious manner, because $P_t$ pd-tiles the group $\Z_{6p}^4\times \Z_{6q}$. Indeed, if $1_{P_t}\ast w=1_{\Z_{6p}^4\times \Z_{6q}}$ is a pd-tiling,
then we can identify $w$ with a nonnegative function defined on the box $\{0, 1,\dots 6p-1\}^4\times \{0, 1, \dots, 6q-1\}$, and letting $w(k)= w\ast 1_{(6pk\Z)^4 \times (6qk\Z)}$ we obtain that   $1_{P(k)}\ast w(k)=1_{\Z^5}$ is a pd-tiling of $\Z^5$.

\medskip

Finally, by the results of \cite{tao} and \cite{tnos}, the set $C(k)=P(k)+[0,1)^5$ is neither a tile, nor spectral in $\R^5$, while it pd-tiles $\R^5$ as $C(k)\ast w(k)=1_{\R^5}$. In particular, the mass of $w(k)$ at the origin is 1, and removing that mass we get a weak tiling of the complement of $C(k)$ by the set $C(k)$. 

\subsection*{Acknowledgments}
The research was partly carried out at the Erd\H os Center, R\'enyi Institute, in the framework of the semester "Fourier analysis and additive problems".

G.K. was supported by the Hungarian National Foundation for
Scientific Research, Grants No. K146922, FK 142993 and by the J\'anos Bolyai Research Fellowship of the Hungarian Academy of Sciences.

I.L. is supported by the Israel Science Foundation (Grant 607/21).

M.M. was supported by the Hungarian National Foundation for Scientific Research, Grants No. K132097, K146387. 

G.S. was supported by J\'anos Bolyai Research  Scholarship and by the Hungarian National Foundation for
Scientific Research, Grants: 138596 and SNN 132625.

\noindent
{\sc Gergely Kiss:}\\
HUN-REN Alfr\'ed R\'enyi Mathematical Institute\\
Re\'altanoda utca 13-15, H-1053, Budapest, Hungary\\
and\\
Corvinus University of Budapest, Department of Mathematics \\
Fővám tér 13-15, Budapest 1093, Hungary,\\
E-mail: {\tt kiss.gergely@renyi.hu}

\noindent{{\sc Itay Londner:}  Department of Mathematics, Faculty of Mathematics and Computer Science, Weizmann Institute of Science, Rehovot 7610001, Israel (ORCID ID: 0000-0003-3337-9427)}\\
E-mail: {\tt itay.londner@weizmann.ac.il}

\noindent
{\sc M\'at\'e  Matolcsi:}\\
HUN-REN Alfr\'ed R\'enyi Mathematical Institute\\
Re\'altanoda utca 13-15, H-1053, Budapest, Hungary\\
and\\
Department of Analysis and Operations Research,
Institute of Mathematics,
Budapest University of Technology and Economics,
M\H uegyetem rkp. 3., H-1111 Budapest, Hungary.
E-mail: {\tt matomate@renyi.hu}

\noindent{{\sc G\'abor Somlai:}  \\
ELTE-TTK, Institute of Mathematics \\ P\'azm\'any P\'eter s\'et\'any 1/C, Budapest, Hungary, H-1117},\\
E-mail: {\tt gabor.somlai@ttk.elte.hu}

\end{document}